\documentclass[12pt]{amsart}
\usepackage{amsthm,amsmath,xspace,url,gensymb,amsfonts}
\usepackage{graphicx,overpic}
\usepackage{hyperref}
\usepackage{enumerate}
\usepackage{color}
\usepackage{centernot}
\usepackage{tikz}
\usetikzlibrary{calc,through,backgrounds}
\usetikzlibrary{arrows,positioning,shapes.geometric}

\definecolor{grey}{gray}{0.6}

\newcounter{intege}
\newcommand{\ngon}[8]{ 

  \foreach \t in {1,...,#1} {
    \coordinate (#7\t) at ($#2+(#8-\t*360/#1:#3)$);
  }
  \setcounter{intege}{1}
  \pgfmathsetcounter{intege}{1}
  \foreach \object/\z in {#6}{
    \node at (#7\theintege) {\z};
    \pgfmathsetcounter{intege}{\theintege+1}
    \setcounter{intege}{\theintege}
  }
  \foreach \x/\y/\z in {#4}{
    \draw[thick,\z,shorten <=#5pt, shorten >=#5pt] {(#7\x)--(#7\y)};
  }
  \foreach \t in {1,...,#1} {
    \coordinate (#7\t) at ($#2+(#8-\t*360/#1:1.15*#3)$);
  }
  \setcounter{intege}{1}
  \pgfmathsetcounter{intege}{1}
  \foreach \object/\z in {#6}{
    \node[inner sep=0pt] at (#7\theintege) {$\object$};
    \pgfmathsetcounter{intege}{\theintege+1}
    \setcounter{intege}{\theintege}
  }
}

\def\x{$\bullet$}
\def\o{$\circ$}

\title[The Cyclic Sieving Phenomenon for non-crossing forests]{The Cyclic Sieving Phenomenon for non-crossing forests}

\author{Stefan Kluge}
\email{kluge.ish@web.de}

\keywords{cyclic sieving, non-crossing forests}

\newtheorem{thm}{Theorem}[section]
\newtheorem{lem}[thm]{Lemma}

\newtheorem{map}[thm]{Mapping}
 
\theoremstyle{definition}
\newtheorem{dfn}[thm]{Definition}
 
\theoremstyle{remark}

\newtheorem*{eg*}{Example}

\newcommand{\Set}[1]{\ensuremath{\mathcal{#1}}}  
\newcommand{\Dfn}[1]{\emph{#1}}                  

\newcommand{\size}[1]{\left\lvert #1\right\rvert}

\newcommand{\qi}[2][q]{[#2]_{#1}}
\newcommand{\qbinom}[3][q]{\genfrac{[}{]}{0pt}{}{#2}{#3}_{#1}}

\DeclareMathOperator{\first}{first}
\DeclareMathOperator{\last}{last}

\begin{document}
\maketitle

\begin{abstract}
	In this paper we prove that the set of non-crossing forests together with a cyclic group acting on it by rotation and a natural q-analogue of the formula for their number exhibits the cyclic sieving phenomenon, as conjectured by Alan Guo.
\end{abstract}

\section{Introduction}
\label{sec:introduction}

\subsection{Non-crossing trees and forests}
	A \Dfn{non-crossing graph} on $n$ vertices is a graph whose vertices are arranged on a circle and whose edges are straight line segments that do not cross. A tree is a connected acyclic graph and a forest is an acyclic graph or, alternatively, a graph whose components are trees. The number of non-crossing forests on $n$ vertices with $k$ components is
	\begin{equation}
	\label{eqn:numForests}
		f_{n,k}=\frac{1}{2n-k}\binom{n}{k-1}\binom{3n-2k-1}{n-k}.
	\end{equation}
	A proof using Lagrange inversion can be found in \cite{FlajoletNoy}.

\begin{figure}
  \centering
\begin{tikzpicture}[auto]
  \coordinate (allo) at (0,0);
  \ngon{12}{(allo)}{2.5}
  {1/2/black,1/8/black,3/7/black,4/7/black,9/11/black}
  {2}{1/\x,2/\o,3/\o,4/\o,5/\o,6/\o,7/\o,8/\o,9/\o,10/\o,11/\o,12/\o}
  {first}{120}
\end{tikzpicture} 
\caption{A non-crossing forest on $n=12$ vertices with $k=7$
  components.}
  \label{fig:f2}
\end{figure}
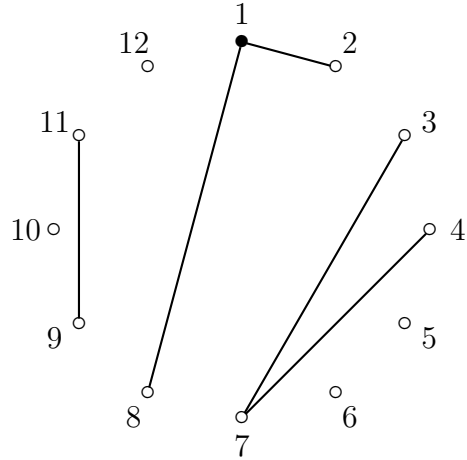

	We obtain a natural q-analogue of ~\eqref{eqn:numForests}	by replacing integers, factorials and binomial coefficients in $f_{n,k}$ by their q-analogues:
	\begin{equation}
		f_{n,k}(q)=\frac{1}{\qi{2n-k}}\qbinom{n}{k-1}\qbinom{3n-2k-1}{n-k},
	\end{equation}
	where
	$$\qi{n}=1+q+q^2+\cdots q^{n-1}=\frac{1-q^n}{1-q}$$
	$$\qi{n}!=\qi{n}\qi{n-1}\cdots\qi{1}$$
	$$\qbinom{n}{k}=\frac{\qi{n}!}{\qi{k}!\qi{n-k}!}.$$

\subsection{Cyclic Sieving Phenomenon}
In 2004 Victor Reiner, Denis Stanton and Dennis White introduced the cyclic sieving phenomenon in ~\cite{RSW}. To define this concept we need a finite set $\Set X$, a cyclic group $C$ of order $n$ acting on $\Set X$ and a polynomial with nonnegative integer coefficients $X(q)$.

\begin{dfn}
\label{dfn:CSP}
	The triple $(\Set X,C,X(q))$ exhibits the \Dfn{cyclic sieving phenomenon} if for all $c\in C$ we have
	$$X(\omega_{o(c)})=\size{\Set X^c},$$
	where $o(c)$ denotes the order of $c$ in $C$, $\omega_d$ is a $d$\textsuperscript{th} primitive root of unity and
	$\Set X^c=\left\{x\in \Set X:c(x)=x\right\}$ denotes the set of fixed points of $\Set X$ under the action of $c\in C$.
\end{dfn}

\begin{dfn}
$\Set F_{n,k}$ is the set of non-crossing forests on $n$ vertices with $k$ components. $\Set F_{n,k}^d$ is the subset of $\Set F_{n,k}$ which contains forests which are invariant under rotation by $\frac{2\pi}{d}$ which we call \Dfn{$d$-invariant} forests.
\end{dfn}

	The following statement is the main theorem of this article. It was conjectured by Alan Guo \cite{Guo}.

\begin{thm}
\label{thm:CSPForests}
  Let $\langle\rho\rangle$ be the cyclic group of order $n$ acting on $\Set F_{n,k}$ by rotation and the polynomial $f_{n,k}(q)$ defined as above. Then $\left(\Set{F}_{n,k},\langle\rho\rangle,f_{n,k}(q)\right)$ exhibits the cyclic sieving phenomenon.
\end{thm}

In the following sections we verify this by direct calculation of the left and right hand side in the condition of Definition~\ref{dfn:CSP}.

\section{the left hand side: evaluating the polynomial}
\label{sec:qEval}
This section is devoted to the proof of the following Lemma which gives the evaluation of $f_{n,k}(q)$ at roots of unity.

\begin{lem}
\label{lem:qEval}
Suppose $d\mid n$. And let $\omega$ be a $d$\textsuperscript{th} primitive root of unity
\begin{itemize}
	\item If $d=1$, then $f_{n,k}(\omega)=f_{n,k}$.
	\item If $d\ge 2$, $d\mid k$, let $n'=\frac{n}{d}$, $k'=\frac{k}{d}$, then\\
	$f_{n,k}(\omega)=\left(n'-k'+1\right)f_{n',k'}$.
	\item If $d=2$, $k$ odd, let $n'=\frac{n}{d}$, $k'=\frac{k+1}{d}$, then\\
	$f_{n,k}(\omega)=\binom{n'}{k'-1}\binom{3n'-2k'}{n'-k'}$.
	\item Otherwise $f_{n,k}(\omega)=0$.
\end{itemize}
\end{lem}

A useful tool for the evaluation of q-binomial coefficients at roots of unity is the q-Lucas theorem:
\begin{lem}[$q$-Lucas theorem]
\label{lem:qLucas}
	Let a and b be nonnegative integers and $\omega$ a $d$\textsuperscript{th} primitive root of unity. Then
	$$\qbinom[\omega]{a}{b}=
		\binom{\lfloor \frac{a}{d} \rfloor} {\lfloor \frac{b}{d} \rfloor}
		\qbinom[\omega]{a-d\lfloor \frac{a}{d} \rfloor} {b-d\lfloor \frac{b}{d} \rfloor}.
	$$
	In particular, if $d$ divides $b$ then
	$$
		\qbinom[\omega]{a}{b}=\binom{\lfloor \frac{a}{d} \rfloor}{\frac{b}{d}}.
	$$
\end{lem}
\begin{proof}
A proof can be found in ~\cite{Sagan}.
\end{proof}

We will also frequently use the following Lemma whose proof is a straightforward calculation.
\begin{lem}
\label{lem:qFacts}
	Let $a$ and $b$ be positive integers and $\omega$ a $d$\textsuperscript{th} primitive root of unity. Then
	\begin{enumerate}
		\item $\qi{a}$ has a simple zero at $\omega$ if and only if $d\ne 1$ and $d\mid a$.
		\item $\qi[\omega]{a}=1$ if $a\equiv 1\bmod d$. 
		\item if $a\equiv b\bmod d$ then
			$$\frac{\qi[\omega]{a}}{\qi[\omega]{b}}=
			\begin{cases}
				&\text{$\frac{a}{b}$ if $a\equiv b\equiv 0\bmod d$}\\
				&\text{1 if $a\equiv b \not\equiv 0\bmod d$}\\
			\end{cases}$$
	\end{enumerate}
\end{lem}

\begin{proof}[Proof of Lemma~\ref{lem:qEval}]
The case $d=1$ follows directly from the definition of the $q$-analogues. For $d\ge 2$ the evaluation of $f_{n,k}(q)$ at $\omega$ splits in two subcases.

\subsection*{The case where $d$ divides $k$}
	Let $n'=\frac{n}{d}$ and $k'=\frac{k}{d}$. Then
	\begin{align*}
		f_{n,k}(\omega)
			&=\frac{1}{\qi[\omega]{2n-k}}
				\qbinom[\omega]{n}{k-1}
				\qbinom[\omega]{3n-2k-1}{n-k}\\
			&=\frac{\qi[\omega]{k}}{\qi[\omega]{2n-k}}
				\frac{1}{\qi[\omega]{n+1}}
				\qbinom[\omega]{n+1}{k}
				\qbinom[\omega]{3n-2k-1}{n-k}\\
	\end{align*}
	Applying the $q$-Lucas Theorem~\ref{lem:qLucas} to the $q$-binomial coefficients and the facts from Lemma~\ref{lem:qFacts} to the factor in front, we obtain
	\begin{align*}
		&\frac{k'}{2n'-k'}
			\binom{n'}{k'}
			\binom{3n'-2k'-1}{n'-k'}\\
		&=(n'-k'+1)
			\frac{1}{2n'-k'}
			\binom{n'}{k'-1}
			\binom{3n'-2k'-1}{n'-k'}\\
		&=(n'-k'+1)f_{n',k'}\\
	\end{align*}
	
\subsection*{The case where $d$ does not divide $k$}
	We have to distinguish two subcases, depending on whether $d$ divides $k-1$ or not.	If $d$ does not divide $k-1$ applying the $q$-Lucas Theorem to the first $q$-binomial coefficient we obtain
	$$
		\qbinom[\omega]{n}{k-1}=\binom{\frac{n}{d}}{\lfloor\frac{k-1}{d}\rfloor}\qbinom[\omega]{0}{k-1-d\lfloor\frac{k-1}{d}\rfloor}=0,
	$$
	because $k-1-d\lfloor\frac{k-1}{d}\rfloor\ne 0$. Since $\qi[\omega]{2n-k}\ne 0$ the polynomial vanishes at $\omega$.

	If $d$ divides $k-1$ the $q$-Lucas Theorem shows that
	$$\qbinom[\omega]{n}{k-1}=\binom{\frac{n}{d}}{\frac{k-1}{d}}.$$
	We can rewrite the remaining as follows:
	\begin{align*}
		&\frac{1}{\qi[\omega]{2n-k}}
			\qbinom[\omega]{3n-2k-1}{2n-k-1}\\
		&=\frac{1}{\qi[\omega]{2n-k}}
			\frac{\qi[\omega]{3n-2k-1}\cdots\qi[\omega]{n-k+1}}{\qi[\omega]{2n-k-1}\cdots\qi[\omega]{1}}\\
		&=\frac{1}{\qi[\omega]{2n-k}}
			\frac{1}{\qi[\omega]{2n-k-1}}
			\frac{\qi[\omega]{3n-2k-1}\cdots\qi[\omega]{n-k+2}}{\qi[\omega]{2n-k-2}\cdots\qi[\omega]{1}}
			\frac{\qi[\omega]{n-k+1}}{1}\\
		&=\frac{\qi[\omega]{n-k+1}}{\qi[\omega]{2n-k}\qi[\omega]{2n-k-1}}
			\frac{\qi[\omega]{3n-2k-1}\cdots\qi[\omega]{n-k+2}}{\qi[\omega]{2n-k-2}\cdots\qi[\omega]{1}}\\
	\end{align*}
	
	In the second fraction the factors above each other have all the same remainder modulo $d$. Hence, by Lemma~\ref{lem:qFacts}, this fraction is non-zero. In the first fraction we evaluate the factors separately:
	$$\qi[\omega]{n-k+1}=\qi[\omega]{n-(k-1)}=0$$
	$$\qi[\omega]{2n-k}=1$$
	$$\qi[\omega]{2n-k-1}
	\begin{cases}
	\text{$=0$, if $d\mid k+1$}\\
	\text{$\ne 0$ else}\\
	\end{cases}$$
	Altogether we have $f_{n,k}(\omega)\ne 0$, if and only if $d\mid k-1$ and $d\mid k+1$ which is only possible if $d=2$. In this case
	$$f_{n,k}(\omega)=\binom{\frac{n}{2}}{\frac{k-1}{2}}\binom{\frac{3n-2k-2}{2}}{\frac{n-k-1}{2}}$$
	as desired.
\end{proof}

\section{The right hand side: counting forests invariant under rotation}

\begin{lem}
\label{lem:disjoint}
	Let $\Phi\in\Set F_{n,k}^d$ and let $\rho$ be the rotation by $\frac{2\pi}{d}$ of $\Phi$. Then $\rho$ maps every tree $\tau$ in $\Phi$ onto a copy of itself, the vertices being disjoint from those of $\tau$, or onto itself. The latter case only occurs if $d=2$ and $k$ is odd.
\end{lem}

\subsection*{The case where $d$ divides $k$}

\begin{dfn}
\label{dfn:distance}
	For two vertices $u$ and $v$ the \Dfn{distance $\ell(u, v)$ from $u$ to $v$} is the number of consecutive vertices when going clockwise from $u$ to $v$. Equivalently, when the vertices are labelled clockwise from $1$ to $n$, the distance is the unique number in $\{1,2,\dots,n\}$ which is congruent to $v-u+1\bmod n$.
\end{dfn}

\begin{dfn}
\label{dfn:navigation}
	Let $u$ and $v$ be vertices with $\ell(u, v)<\ell(v, u)$. Then \Dfn{$u$ is in front of $v$} and \Dfn{$v$ is behind $u$}. In particluar, if $\ell(u,v)=1$ \Dfn{$u$ is the predecessor of $v$} and \Dfn{$v$ is the successor of $u$}
\end{dfn}

\begin{dfn}
\label{dfn:goodVertex}
Let us label the vertices of a forest in $\Set F_{n',k'}$ clockwise from $1$ to $n'$ beginning with the base vertex. For any tree not containing the base vertex $1$ we call the vertex with minimal label a \Dfn{bad vertex}. All $n'-k'+1$ other vertices are \Dfn{good}.
\end{dfn}

\begin{lem}
\label{lem:size}
  Let $\tau$ be a tree in a forest $\Phi\in\Set F_{n,k}^d$ and let $\rho(\tau)$ be its image under rotation by $\frac{2\pi}{d}$. Then there is a unique vertex $v\in\tau$ such that there are no vertices of $\tau$ between $v$ and any vertex of $\rho(\tau)$. Subsequently, $v$ is called the \Dfn{last vertex} of $\tau$.

 Similarly, let $\rho^{-1}(\tau)$ be the preimage of $\tau$ under rotation by $\frac{2\pi}{d}$. Then there is a unique vertex $u\in\tau$ such that there are no vertices of $\tau$ between any vertex of $\rho^{-1}(\tau)$ and $u$. Subsequently, $u$ is called the \Dfn{first vertex of $\tau$}.

	Thus, between the first and the last vertex of a tree $\tau$ there cannot be any vertex of an image of $\tau$. In particular, the number of components of $\Phi$ is a multiple of $d$ and the distance between the first and the last vertex of any tree $\tau$ is at most $n'=\frac{n}{d}$.
\end{lem}
\begin{proof}
  Follows directly from Lemma~\ref{lem:disjoint} together with the fact that $\Phi$ is non-crossing.
\end{proof}

\begin{lem}
\label{lem:numInvForests}
	Suppose $d\mid n$. Let $n'=\frac{n}{d}$ and $k'=\frac{k}{d}$. If $d$ divides $k$ the number of non-crossing forests with $n$ vertices and $k$ components invariant under rotation by $\frac{2\pi}{d}$ is $(n'-k'+1)f_{n',k'}$.
\end{lem}

To prove this we will show that the following mapping is in fact a bijection.

\begin{map}
\label{map:dBijection}
Suppose $d\ge2$ and $d$ divides $k$ and let $n'=\frac{n}{d}$, $k'=\frac{k}{d}$. Abusing notation, let $\Set F_{n,k}\times G$ be the set of pairs $(\phi,v)$ such that $\phi\in\Set F_{n,k}$ and $v$ is a good vertex in $\phi$.

{\bf Construction:} The function $C_d:\Set F_{n',k'}\times G\rightarrow\Set F_{n,k}^d$ maps $(\phi,v)\in\Set F_{n',k'}\times G$ onto a $d$-invariant forest $\Phi\in\Set F_{n,k}^d$ as follows:

We arrange the vertices of $\phi$ in a list beginning with the vertex $v$ such that the predecessor of $v\in\phi$ is the last vertex in the list. Then we construct a new list on $n$ vertices by placing $d$ copies of this list next to each other such that for $1\le i<d$ 
the first vertex in the $(i+1)$\textsuperscript{st} sublist is the right neighbour of the last vertex in the $i$\textsuperscript{th} sublist.

Now, we arrange the vertices of this list on a circle such that the first vertex of the list becomes the successor of the last vertex in the list. Finally, any of the vertices corresponding to the base vertex of $\phi$ becomes the base vertex of the so constructed $d$-invariant forest.

{\bf Decomposition:} the function $D_d:\Set F_{n,k}^d\rightarrow\Set F_{n',k'}\times G$ maps a $d$-invariant forest $\Phi\in\Set F_{n,k}^{d}$ onto a pair $(\phi,v)\in\Set F_{n',k'}\times G$ as follows:

We choose the list of $n'$ consecutive vertices in $\Phi$ beginning with the first vertex $v$ in front of the base vertex such that there is no edge in $\Phi$ connecting a vertex in this list with a vertex not in this list. Lemma~\ref{lem:exist} shows that such a vertex exists and that the base vertex is in this list.

We then arrange the vertices of this list on a circle. The vertex corresponding to the base vertex in $\Phi$ becomes the base vertex in the constructed forest.
\end{map}

%
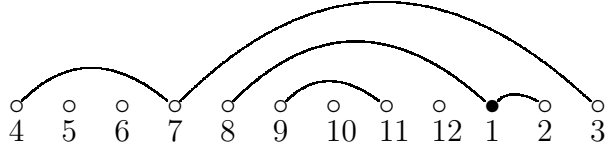
\begin{figure}
  \setlength{\unitlength}{20pt}
  \begin{picture}(10,4)(0,-1)
    \put(0,-0.5){\hbox{$4$}}\put(0,0){\o}
    \put(1,-0.5){\hbox{$5$}}\put(1,0){\o}
    \put(2,-0.5){\hbox{$6$}}\put(2,0){\o}
    \put(3,-0.5){\hbox{$7$}}\put(3,0){\o}
    \put(4,-0.5){\hbox{$8$}}\put(4,0){\o}
    \put(5,-0.5){\hbox{$9$}}\put(5,0){\o}
    \put(6,-0.5){\hbox{$10$}}\put(6,0){\o}
    \put(7,-0.5){\hbox{$11$}}\put(7,0){\o}
    \put(8,-0.5){\hbox{$12$}}\put(8,0){\o}
    \put(9,-0.5){\hbox{$1$}}\put(9,0){\x}
    \put(10,-0.5){\hbox{$2$}}\put(10,0){\o}
    \put(11,-0.5){\hbox{$3$}}\put(11,0){\o}
    \qbezier(0.25,0.25)(1.5,1.5)(3,0.25)
    \qbezier(3.25,0.25)(7.0,4)(11,0.25)
    \qbezier(4.25,0.25)(6.5,2.5)(9,0.25)
    \qbezier(5.25,0.25)(6,1)(7,0.25)
    \qbezier(9.25,0.25)(9.5,0.5)(10,0.25)
  \end{picture}
  \caption{The non-crossing forest of Figure~\ref{fig:f2} drawn as a
    list with good vertex $4$.}
  \label{fig:l2}
\end{figure}

\begin{figure}
  \centering
\begin{tikzpicture}[auto]
  \coordinate (allo) at (0,0);
  \ngon{24}{(allo)}{2.5}
  {1/2/black,3/19/black,4/7/black,7/15/black,8/13/black,9/11/black,13/14/black,16/19/black,20/1/black,21/23/black}
  {2}{1/\x,2/\o,3/\o,4/\o,5/\o,6/\o,7/\o,8/\o,9/\o,10/\o,11/\o,12/\o,
13/\o,14/\o,15/\o,16/\o,17/\o,18/\o,19/\o,20/\o,21/\o,22/\o,23/\o,24/\o}
  {first}{105}
\end{tikzpicture} 
\caption{The image of the non-crossing forest of Figure~\ref{fig:f2}
  under $C_2$ with good vertex $4$.}
  \label{fig:F2}
\end{figure}
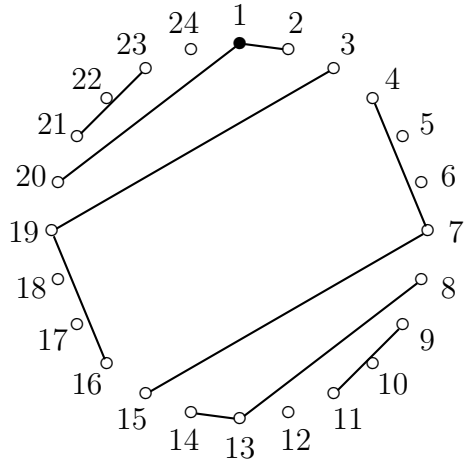

\begin{lem}
\label{lem:exist}
$D_d$ is well defined.
\end{lem}
\begin{proof}
We need to prove that in every $\Phi\in\Set F_{n,k}$ a list of list of $n'=\frac{n}{d}$ consecutive vertices exists, such that no edge in $\Phi$ connects a vertex corresponding to a vertex in this list with a vertex not in this list.

We can construct such a list containing the base vertex as follows: Draw a straight line $r$ from the center of the circle to the base vertex. If $r$ crosses an edges in $\Phi$, then there is an unique tree $\tau$ containing the edge $e$ such that the crossing of $e$ and $r$ is the closest crossing to the center.

We then choose the list of consecutive vertices such that $\first(\tau)$ is the first vertex in the list and $\last(\tau)$ is the last vertex in the list. By construction the base vertex $1$ is between $\first(\tau)$ and $\last(\tau)$. Hence, the label of $\last(\tau)$ is smaller than the label of $\first(\tau)$ and thus $\first(\tau)$ is a good vertex.

If $r$ does not cross any edge, let $\tau$ be the tree containing the base vertex. In this case $\first(\tau)$ is automatically a good vertex.

By Lemma~\ref{lem:size}, the number of vertices in this list is at most $n'$.

If it is smaller than $n'$, choose the tree $\sigma$ in $\Phi$ where $\first(\sigma)$ is the successor of the vertex corresponding to the last vertex in the list. Since $\sigma$ and the image of $\tau$ under rotation are disjoint, the list beginning at $\first(\tau)$ and ending at $\last(\sigma)$ contains not more than $n'$ vertices.

Then we repeat the last step till our list contains $n'$ vertices.
\end{proof}

\begin{lem}
\label{lem:bijection}
$C_d$ is a bijection between $\Set F_{n,k}^{d}$ and $\Set F_{n',k'}\times G$.
\end{lem}

\begin{proof}
We will prove that the $C_d$ and $D_d$ are compositional inverses of each other.

Let $\Phi\in\Set F_{n,k}^d$ and $D_d(\Phi)=(\phi,v)$. Then $C_d(D_d(\Phi))=\Phi$ since both the list in the decomposition and the list in the construction begin with $v$ and represent the same forest $\phi$.

Conversely, let $(\phi,v)\in\Set F_{n',k'}\times G$ and $C_d(\phi,v)=\Phi$. To prove that $D_d(C_d(\phi,v))=(\phi,v)$ we have to show that the vertex returned by the decomposition is $v$. By construction of $\Phi$ no edge connects a vertex in the list of $n'$ consecutive vertices beginning with a vertex corresponding to $v$ with a vertex not in this list. Thus, it remains to show that $v$ is the first vertex in front of the base vertex having this property.

This is certainly the case if $v\in\phi$ is in the tree containing the base vertex.

Else, if $v$ is in a tree $\tau$ not containing the base vertex, then $v$ is not the vertex with minimal label in $\tau$, since $v$ is a good vertex. Thus, $\tau$ contains vertices with smaller label than $v$, and $\tau$ contains vertices in front of and behind the base vertex. Since a list of $n'$ consecutive vertices beginning with a vertex between $v$ and the base vertex contains the vertices of $\tau$ which are behind the base vertex, a path in $\Phi$ connects a vertex in this list with $v$ which is not in the list. Hence, $v$ is the first vertex in front of the base vertex with the required property such that both the lists in the construction and in the decomposition are beginning with $v$ and representing the same forest $\phi$.
\end{proof}

\begin{proof}[Proof of Lemma~\ref{lem:numInvForests}]
We have a bijection between $\Set F_{n,k}^d$ and $\Set F_{n',k'}\times G$. Since for each of the $f_{n',k'}$ elements in $\Set F_{n',k'}$ we have $n'-k'+1$ good vertices, the number of elements in in $\Set F_{n,k}^d$ is $f_{n,k}^d=(n'-k'+1)f_{n',k'}$.
\end{proof}

\subsection*{The case where $d=2$ and $k$ odd}
\begin{lem}
\label{lem:num2InvForests}
Let $d=2$, $n$ even and $k$ odd and let $n'=\frac{n}{2}$ and $k'=\frac{k+1}{2}$. Then the number of non-crossing forests under rotation by $\pi$ is $(3n'-2k')f_{n',k'}$
\end{lem}

The proof is similar to the one of Lemma~\ref{lem:numInvForests}. We define a bijection between $\Set F_{n,k}^2$ where $k$ is odd and $\Set F_{n',k'}$ with additional information.

\begin{map}
Abusing notation let $\Set F_{n',k'}\times G$ be the set forests $\phi\in\Set F_{n',k'}$ with either one of the $n'$ vertices marked or one of the $n'-k'$ edges and an incident vertex marked.

{\bf Construction:} the function $C_2:\Set F_{n',k'}\times G\rightarrow\Set F_{n,k}^2$ maps a forest $\phi$ with coloring onto a $2$-invariant forest $\Phi$ as follows:

	We draw a diameter from top to bottom in a circle. Then we arrange the vertices of $\phi$ clockwise on the right hand side of the circle such that the marked vertex $v$ becomes the upper vertex of the diameter.

	We have to distinguish whether an edge is marked or not. If no edge is marked, we ignore the lower vertex of the diameter and draw all edges incident to $v$ in $\phi$ such that they are incident to the upper vertex of the diameter.
	
	In contrast, if an edge is marked, going clockwise starting at $v$ we draw all vertices adjacent to $v$ till we reach the vertex incident to the marked edge. This vertex becomes the first vertex adjacent to the lower vertex of the diameter. All consecutive vertices adjacent to $v$ become adjacent to the lower vertex of the diameter. Thus, the marked vertex is the first edge drawn incident to the lower vertex.
	
	Finally, the remainder of the forest is completed so that the
        result becomes symmetric.

{\bf Decomposition:} the function $D_2:\Set F_{n,k}^2\rightarrow\Set F_{n',k'}\times G$ maps a $2$-invariant forest $\Phi$ with odd number of components onto a forest $\phi\in\Set F_{n',k'}$ with coloring as follows:

	The diameter separates the circle into two segments. We call the side with the base vertex the right hand side. Then we disregard the vertices in the left hand side of $\Phi$ such that only the diameter and the right hand side remain. Consequently, we call the vertices of the diameter the upper vertex and the lower vertex. 

	Then going clockwise starting at the upper vertex we choose the first vertex adjacent to the lower vertex and mark the edge between the two vertices.

	Finally, we obtain a $\phi$ by contracting the diameter which becomes the marked vertex.
\end{map}

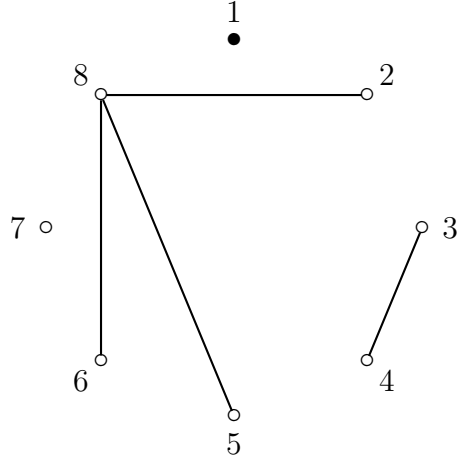
\begin{figure}
  \centering
\begin{tikzpicture}[auto]
  \coordinate (allo) at (0,0);
  \ngon{8}{(allo)}{2.5}
  {3/4/black,5/8/black,6/8/black,8/2/black}
  {2}{1/\x,2/\o,3/\o,4/\o,5/\o,6/\o,7/\o,8/\o}
  {first}{135}
\end{tikzpicture} 
\caption{A non-crossing forest on $n=8$ vertices with $k=3$
  components.}
  \label{fig:f3}
\end{figure}

\begin{figure}
  \centering
\begin{tikzpicture}[auto]
  \coordinate (allo) at (0,0);
  \ngon{16}{(allo)}{2.5}
  {3/4/black,5/16/black,6/16/black,8/16/black,8/10/blue,8/13/blue,8/14/blue,11/12/blue,16/2/black}
  {2}{1/\x,2/\o,3/\o,4/\o,5/\o,6/\o,7/\o,8/\o,9/\o,10/\o,11/\o,12/\o,13/\o,14/\o,15/\o,16/\o}
  {first}{112.5}
\end{tikzpicture} 
\caption{The image of the non-crossing forest of Figure~\ref{fig:f3}
  when marking vertex $8$.}
  \label{fig:F3a}
\end{figure}
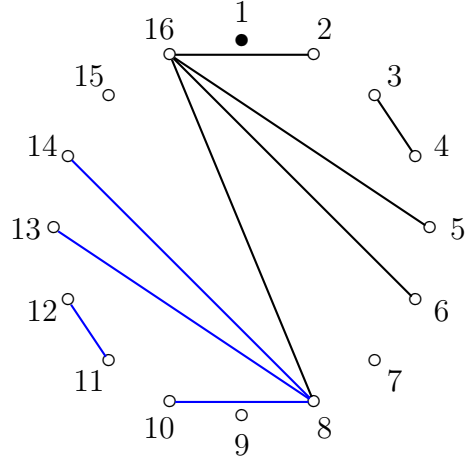

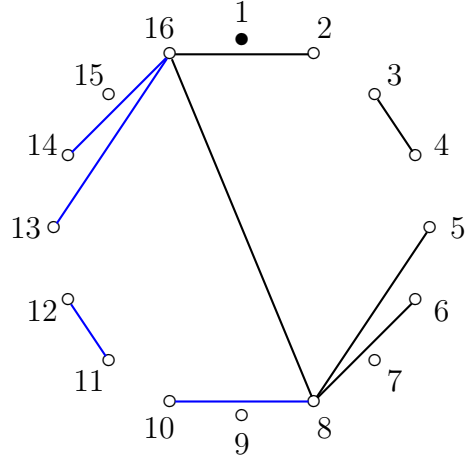
\begin{figure}
  \centering
\begin{tikzpicture}[auto]
  \coordinate (allo) at (0,0);
  \ngon{16}{(allo)}{2.5}
  {3/4/black,5/8/black,6/8/black,8/10/blue,8/16/black,11/12/blue,13/16/blue,14/16/blue,16/2/black}
  {2}{1/\x,2/\o,3/\o,4/\o,5/\o,6/\o,7/\o,8/\o,9/\o,10/\o,11/\o,12/\o,13/\o,14/\o,15/\o,16/\o}
  {first}{112.5}
\end{tikzpicture} 
\caption{The image of the non-crossing forest of Figure~\ref{fig:f3}
  when marking vertex $8$ and the edge $(5,8)$.}
  \label{fig:F3b}
\end{figure}

\begin{lem}
$C_2$ is a bijection $C_2$ and $D_2$ are compositional inverses of each other.
\end{lem}

\begin{proof}[Proof of Lemma~\ref{lem:num2InvForests}]
	$C_2$ is a bijection between $\Set F_{n,k}^2$ and $\Set F_{n',k'}\times G$.	For all $\phi\in\Set F_{n',k'}$ we have $(3n'-2k')$ different colorings since we have $n'$ options to mark a vertex and $2(n'-k')$ options to mark an edge and an incident vertex. Hence, the number of of elements in $\Set F_{n,k}^2$ is $f_{n,k}^2=(3n'-2k')f_{n',k'}$
\end{proof}

\section{Polynomiality and nonnegativity of $f_{n,k}(q)$}
\label{sec:polynom}
We prove that $f_{n,k}(q)$ is a polynomial with nonnegative integer coeffients, as required in Definition~\ref{dfn:CSP}. Let $a(q)=\qbinom{n}{k-1}$, $b(q)=\qbinom{3n-2k-1}{n-1}$ and $c(q)=\qi{2n-k}$ so that $f_{n,k}(q)=\frac{a(q)b(q)}{c(q)}$ and let $\omega$ be a $d$\textsuperscript{th} primitive root of unity. Then $f_{n,k}(q)\in\mathbb{Q}\left[q\right]$, if for all $d\in\mathbb{N}$ the multiplicity of a zero at $\omega$ in $a(q)b(q)$ is greater or equal than in $c(q)$. Note that $c(q)$ can have only simple zeros by Lemma ~\ref{lem:qFacts}. As a result of Section~\ref{sec:qEval} we know that $f_{n,k}(\omega)$ is an integer if $d$ divides $n$, and in this case the multiplicity of a zero at $\omega$ is not smaller in $a(q)b(q)$ than in $c(q)$. Since $c(q)$ has a simple zero if and only if $d\mid 2n-k$, it remains to consider the case $d\centernot\mid n$ and $d\mid 2n-k$. In this case $d\centernot\mid k$ and $n\centernot\equiv k\bmod d$. Let $n=dq_n+r_n$ and $k=dq_r+r_k$ with $0<r_n,r_k<d$. If $r_k=1$ then $r_n\ne0$, $r_n\ne1$ and by $q$-Lucas Theorem
$$b(\omega)=\binom{\cdots}{\cdots}\qbinom[\omega]{r_n-r_k-1}{r_n-r_k}=0$$
If $r_k>1$
$$a(\omega)=\binom{\cdots}{\cdots}\qbinom{r_n}{r_k-1}=0\text{, if $r_n<r_k-1$}$$
$$b(\omega)=\binom{\cdots}{\cdots}\qbinom[\omega]{d-2}{d-1}=0\text{, if $r_n=r_k-1$}$$
$$b(\omega)=\binom{\cdots}{\cdots}\qbinom[\omega]{r_n-r_k-1}{r_n-r_k}=0\text{, if $r_n>r_k-1$}.$$
Thus, if $c(q)$ vanishes at least one of $a(q)$ and $b(q)$ vanish, too. Since $a(q)$ and $b(q)$ are polynomials in $\mathbb{N}\left[q\right]$ with symmetric, unimodal coefficient sequences, their product is as well symmetric and unimodal. Finally, $f_{n,k}(q)\in\mathbb{N}\left[q\right]$ follows from ~\cite[Proposition 10.1]{RSW}.

\providecommand{\cocoa} {\mbox{\rm C\kern-.13em o\kern-.07em C\kern-.13em
  o\kern-.15em A}}
\providecommand{\bysame}{\leavevmode\hbox to3em{\hrulefill}\thinspace}
\providecommand{\href}[2]{#2}

\end{document}